\documentclass[12pt, reqno]{amsart}
\usepackage{amsmath, amsthm, amscd, amsfonts, amssymb, graphicx, color, dsfont, mathrsfs, enumerate, mathtools, xcolor}
\usepackage{inputenc}
\usepackage[all]{xy}
\usepackage[bookmarksnumbered, colorlinks, plainpages]{hyperref}
\hypersetup{colorlinks=true,linkcolor=red, anchorcolor=green, citecolor=cyan, urlcolor=red, filecolor=magenta, pdftoolbar=true}
\makeatother
\normalsize
\newtheorem{theorem}{Theorem}[section]

\newtheorem{proposition}[theorem]{Proposition}
\newtheorem{corollary}[theorem]{Corollary}
\theoremstyle{definition}

\theoremstyle{remark}

\numberwithin{equation}{section}

\allowdisplaybreaks

\newcommand{\A}{\mathcal A}
\numberwithin{equation}{section}

\begin{document}
\title [Orthogonally additive polynomials on the bidual of   Banach algebras]{Orthogonally additive polynomials on the bidual of  	Banach algebras}
\author[A.A. Khosravi]{AminALLah Khosravi}
\author[H.R. Ebrahimi Vishki]{Hamid Reza Ebrahimi Vishki}
\author[R. Faal]{Ramin  Faal}
\address{ Department of Pure Mathematics, Ferdowsi University of Mashhad, P.O. Box 1159, Mashhad 91775, IRAN.}
\email{amin.khosravi@mail.um.ac.ir}
\address{ Department of Pure Mathematics and  Center of Excellence in Analysis on Algebraic Structures (CEAAS), Ferdowsi University of Mashhad, P.O. Box 1159, Mashhad 91775, IRAN.}
\email{vishki@um.ac.ir}
\address{ Department of Pure Mathematics, Ferdowsi University of Mashhad, P.O. Box 1159, Mashhad 91775, IRAN.}
\email{faal.ramin@yahoo.com}

\subjclass[2010]{Primary: 46G25; Secondary:  46B20; 46B10}
\keywords{{Orthogonally additive homogeneous polynomial; Arens product; Aron-Berner extension; dual Banach algebra}}
\maketitle
\begin{abstract}
 			We say that a Banach algebra $\A$ has $k-$orthogonally additive property ($k-$OA property, for short) if every orthogonally additive $k-$homogeneous polynomial $P:\A\longrightarrow\mathbb C$ can be expressed  in the standard form $P(x)=\langle\gamma, x^k\rangle\ (x\in\A)$, for some $\gamma\in\A^*.$  In this paper we  first investigate the extensions of  a  $k-$homogeneous polynomial from $\A$ to the  bidual $\A^{**},$ equipped with the first  Arens product.  We then study the relationship between  $k-$OA properties of $\A$ and $\A^{**}.$  This relation  is specifically investigated for a dual Banach algebra. Finally we examine our  results for the dual Banach algebra $\ell^1$, with pointwise product,  and  we show that the Banach algebra  $(\ell^1)^{**}$ enjoys $k-$OA property.
\end{abstract}
\section{introduction}
 Let $\A$ be a Banach algebra and let $k\geq 2$ be an integer. Hereafter,  $\A^k$ stands for  the $k-$fold product of $A$. By a $k-$homogeneous polynomial on $\A$ we mean a  map $P:\A\longrightarrow\mathbb C$ for which there exists a  bounded symmetric  $k-$linear form  $\Phi: \A^k\longrightarrow\mathbb  C$
satisfying $P(x) = \Phi(x,\cdots,x)$, for every $x\in\A$. Such a map $\Phi$ is unique  and as a consequence of the polarization formula can be obtained through
\begin{eqnarray}\label{Phi}
\Phi(x_1,\ldots,x_k)=\frac{1}{k!2^k}\sum_{\epsilon_1,\cdots,\epsilon_k=\pm1}\epsilon_1\cdots\epsilon_n P(\epsilon_1 x_1+\cdots+\epsilon_k x_k)\quad (x_1,\ldots,x_k\in\A^k).
\end{eqnarray}
We say that a $k-$homogeneous polynomial $P$ is orthogonally additive  if $P(x+y)=P(x)+P(y)$ whenever $x,y$ are orthogonal, that is, $xy=0=yx, (x,y\in\A)$. It is easy to verify that every functional   $\gamma$ in $\A^*$  induces an $k-$homogeneous polynomial $P_\gamma$ defined by $ { P}_\gamma(x)=\langle\gamma, x^k\rangle$ on $\A,$ which is orthogonally additive. The $k-$homogeneous polynomials of this type are called standard. In this case the unique  symmetric $k-$linear form  associated with $P_\gamma$  (see \eqref {Phi}) is:
\begin{eqnarray}\label{standard}
 \Phi(x_1,\ldots,x_k)=\frac{1}{k!}\sum_{\sigma\in S_k}\langle\gamma, x_{\sigma(1)}\ldots x_{\sigma(k)}\rangle\quad (x_1,\ldots,x_k\in\A^k),
 \end{eqnarray}
where $S_k$ denotes the permutation group of  the set $\{1,\ldots,k\}.$\\

%
  The main question,  however, is whether every  orthogonally additive $k-$homogeneous polynomial on a Banach  algebra   can be expressed  in the standard form. We say that a Banach algebra $\A$ has   $k-$orthogonally additive  property ($k-$OA property, for short), if  every  orthogonally additive $k-$homogeneous polynomial on   $\A$ is in the standard form. $k-$OA property of a wide variety of Banach algebras were investigated by several authors. For the algebra of continuous functions  and Banach
lattices it has been independently studied in \cite{PV} and  \cite{BLL}, respectively.  In \cite{PPV} the question was treated for  a general C$^*$-algebra. It was studied  for the Fourier algebras and the convolution group algebras   on a locally compact group in \cite{AEV1} and  \cite{AEGV}, respectively. For the Banach algebras satisfying the properties $\mathbb A$ and $\mathbb B$ it has been investigated in \cite{ABSV}.  In \cite{V} the question was positively answered for a certain class of Banach function algebras which also promoted  some older results in this direction.

It seems natural to ask whether the $k-$OA property is preserved by extending the homogeneous polynomials to the   bidual spaces. Our main aim here is investigating the relation between the $k-$OA property of a Banach algebra $\A$ with the $k-$OA property of the bidual $\A^{**}.$ For this purpose, we first study the possibility of transferring the $k-$OA property by extending the orthogonally additive $k-$homogeneous polynomials to $\A^{**}$ for a general Banach algebra $\A$. We then restrict ourselves to  the dual Banach algebras and we show that for  a dual Banach algebra $\A$,  the $k-$OA property inherits from $\A^{**}$ to $\A$ itself (Theorem \ref{inherits}). To investigate the converse direction, in Theorem \ref{aks} we  show that this property can be transferred from $\A$ to $\A^{**}$ for  certain class of dual Banach algebras and we  apply  it for the dual Banach algebra $\ell^1$ in Corollary \ref{l}.
\section{Extension of homogeneous polynomials to the bidual spaces}
Extending a  bounded bilinear map on Bnaach spaces  to the bidual spaces was initiated by the pioneer work of Arens \cite{ Ar}, where he  also originated  two, generally different, (Arens) products $\Box$ and $\lozenge$ to the bidual $\A^{**}$ of a Banach algebra $\A.$ A Banach algebra $A$ is called Arens regular when the products $\Box$ and $\lozenge$ coincide on the whole of $\A^{**}$. More information  on the properties of Arens products and related materials  can be found in \cite{D}.  In the realm of multilinear maps, the  most influential  method for extending a  multilinear map to the bidual spaces  is due to Aron and Berner \cite{AB}, who showed that such extensions always exist (see also \cite{DG}, where  the isometric nature of the Aron--Berner extensions is also established).

Before proceeding,  we need to clarify what we mean by extending a homogeneous polynomial to the bidual. Let us recall that every    bounded (symmetric) $k-$linear  form $\Phi:\A^k\longrightarrow\mathbb C$ has $k!$ many, generally different,  Aron--Berner extensions to the bidual spaces.   Indeed, corresponding to each permutation  in $S_k$ there is   an Aron--Berner extension of $\Phi$ to the biduals, however, here and subsequently,  we only deal with the (first) Aron-Berner extension $\overline\Phi:(\A^{**})^k\longrightarrow\mathbb C,$  corresponds to the identity permutation, which is uniquely defined by
\begin{align}\label{ar}
\overline{\Phi}(m_1,\ldots,m_k)=\text{weak}^*\text{-}\lim_{\alpha_1}\cdots\text{weak}^*\text{-}\lim_{\alpha_k}\Phi(a^1_{\alpha_1},\cdots,a^k_{\alpha_k}),
\end{align}
 where $(a^i_{\alpha_i})$ is a net in $\mathcal{A}$,  weak$^*$-converging to $m_i\in\mathcal{A}^{**}$, for each $1\leq i\leq k.$
Moreover, $\overline{\Phi}$ is a bounded $k-$linear form with the same norm as $\Phi$. It should be remarked that $\overline\Phi$  is not symmetric, in general,  (see \cite{KEP}, in which  the authors provided some examples  illustrating this discrepancy for  the case of triple maps and they also investigated some conditions under which certain extensions are symmetric). Although, similar to \eqref{standard}, $\overline \Phi$ can be symmetrized in a natural way by  ${\overline \Phi}^s$ which has the presentation 
\begin{align}\label{standard3}
{\overline \Phi}^s(m_1,\ldots,m_k)=\frac{1}{k!}\sum_{\sigma\in S_k}\overline{\Phi}\left(m_{\sigma(1)},\ldots,m_{\sigma(k)}\right)\quad (m_1,\ldots,m_k\in\A^{**}).
\end{align}
 We denote the $k-$homogeneous polynomial induced by ${\overline \Phi}^s$ (or $\overline\Phi$ ) by $\overline P$, that is, $\overline P(m)={\overline \Phi}^s(m,\ldots,m)=\overline\Phi(m,\ldots,m),$ for each $m\in\A^{**}.$ It naturally extends the $k-$homogeneous polynomial $P:\A\longrightarrow\mathbb C$ induced by $\Phi$. We usually call $\overline P$ as the Aron-Berner extension of $P.$\medskip

It is worthwhile mentioning that the $k-$homogeneous polynomial $\overline P$ does not reflect the  orthogonal additivity of $P$, in general. To clarify this,
 		let  $\gamma\in\A^*$ and let $P_\gamma:\A\rightarrow \mathbb{C}$,  $P_\gamma(x)=\langle\gamma, x^2\rangle, (x\in\A)$, be the standard  $2-$homogeneous polynomial on $\A$. Let $\Phi:\A^2\rightarrow \mathbb{C}$  be the unique symmetric bilinear form associated with $P_\gamma$, which by  \eqref{standard} is given by the role $\Phi(x,y)=\frac{1}{2}\langle\gamma, xy+yx\rangle$ for every $x,y\in\A$. Extending $\Phi$ to the biduals,  we arrive at  \[\overline{\Phi}(m,n)=\frac{1}{2}\langle m\square n+n\lozenge m, \gamma\rangle,\] for every $m,n\in\A^{**}.$ Then it induces the  $2-$homogeneous polynomial $\overline{P_\gamma}(m)=\overline{\Phi}(m,m)=\overline{P_\gamma}(m)=\frac{1}{2}\langle m\square m+m\lozenge m, \gamma\rangle$. Now if  $m\square n=0=n\square m$, then \[\overline{P_\gamma}(m+n)=\overline{P_\gamma}(m)+\overline{P_\gamma}(n)+\frac{1}{2}\langle m\Diamond n+n\Diamond m, \gamma\rangle. \]
 	The latter identity  shows that $\overline{P_\gamma}$ is not orthogonally additive, in general. However, it is orthogonally additive in the case where $\A$ is either commutative (note that we get $n\lozenge m=m\Box n$ when $\A$ is commutative) or Arens regular. In these rich cases we further have that $\overline{P_\gamma}(m)=\langle m^2, \gamma\rangle,$ for each $m\in\A^{**}$, which confirms that $\overline{P_\gamma}$  is also standard.
 	
 	These observations lead us to   the following result that seems to be interesting in its own right. Hereafter, we consider  $\A^{**}$ as a Banach algebra equipped with its first Arens product $\Box$. 
 	\begin{proposition}\label{aa}
	Let $\mathcal{A}$ be a Banach algebra which is either commutative or  Arens regular. If $P_\gamma:\mathcal{A}\rightarrow\mathbb C$ is a standard $k$-homogeneous  polynomial ($\gamma\in\A^*$), then so is its (Aron-Berner)  extension $\overline{P_\gamma}:\mathcal{A}^{**}\rightarrow \mathbb C$ which is also induced by the same $\gamma$. 
\end{proposition}
\begin{proof}
 Let $\Phi:\mathcal{A}^k\rightarrow\mathbb C$ be the unique  symmetric $k-$linear form associated with $P_\gamma$ as presented in \eqref{standard}. In the case where $\A$ is Aren regular, a routine iterated limit procedure shows that  the Aron-Berner extension $\overline{\Phi}:(\A^{**})^k\longrightarrow\mathbb C$ of $\Phi$ has the  presentation
\begin{eqnarray}\label{standard1}
 \overline{\Phi}(m_1,\ldots,m_k)=\frac{1}{k!}\sum_{\sigma\in S_k}\langle m_{\sigma(1)}\square\ldots\square m_{\sigma(k)},\gamma\rangle,\quad (m_1,\ldots,m_k\in\A^{**}).
 \end{eqnarray}
While in the case that  $\A$ is commutative   $\overline\Phi$ enjoys the expression
 \begin{eqnarray}\label{standard2}
 \overline{\Phi}(m_1,\ldots,m_k)=\langle m_1\square\ldots\square m_k,\gamma\rangle,\quad (m_1,\ldots,m_k\in\A^{**}).
 \end{eqnarray}
Therefore,  by \eqref{standard1} and \eqref{standard2}, in the both  cases  the (Aron-Berner) extension $\overline{P_\gamma}$   has the presentation $\overline {P_\gamma}(m)= \overline{\Phi}(m,\ldots,m)=\langle m^k,\gamma\rangle,$ for each $m\in\A^{**}$. In particular, $\overline{P_\gamma}$ is standard.
 \end{proof}

From the discussion before Proposition \ref{aa}, we can see that for a standard $2$-homogeneous polynomial $P$, the  extension $\overline P$ is orthogonally additive whenever $m\square n=n\square m=0$ implies that $m\Diamond n+n\Diamond m=0$. This condition is evident   when $\mathcal{A}$ is either commutative or Arens-regular and Proposition \ref{aa} in addition confirms that $\overline P$ is standard in each case. The next result studies a Banach algebra $\A$ which is neither commutative nor Arens regular whose bidual $\A^{**}$ does not have $2-$OA property.  However, to the best of our knowledge, we still do not know a Banach algebra $\A$ satisfying the conditions of Proposition \ref{ns}, (see the Question at the end of the paper).
\begin{proposition}\label{ns}
Let $\mathcal{A}$ be a Banach algebra with a bounded approximate identity  such that  $m_0\square m_0\neq m_0\Diamond m_0$ for some  $m_0\in \mathcal{A}^{**}$. If for each $m,n\in \mathcal{A}^{**}$ the equality  $m\square n=n\square m=0$ implies $m\Diamond n+n\Diamond m=0,$ then there exists a standard $2-$homogeneous polynomial $P_\gamma$ ($\gamma\in\A^*$) on $\mathcal{A}$ such that  $\overline{P_\gamma}$  is orthogonally additive on $\A^{**}$ but is  not standard. In particular,  $\mathcal{A}^{**}$ has not $2-$OA property.
\end{proposition}
\begin{proof}
By the hypothesis there exists  an element $m_0\in\A^{**}$ such that  $m_0\square m_0\neq m_0\Diamond m_0$.  We  choose a $\gamma\in \mathcal{A}^*$ such that $\langle m_0\square m_0, \gamma\rangle\neq \langle m_0\Diamond m_0,\gamma\rangle$. Let $P_\gamma:\A\longrightarrow\mathbb C$ be the standard $2-$homogeneous polynomial induced by $\gamma$, that is $P_\gamma(x)=\langle\gamma,x^2\rangle,\ (x\in\A)$,  whose associated symmetric bilinear form has the presentation  $\Phi(x,y)=\frac{1}{2}\langle\gamma,(xy+yx)\rangle\ (x,y \in \mathcal{A})$.  We naturally extend $\Phi$ to the Aron-Berner extension $\overline{\Phi}:(\A^{**})^2\longrightarrow\mathbb C$ which (from \eqref{ar}) can be represented by $\overline{\Phi}(m,n)=\frac{1}{2}\langle(m\square n+n\Diamond m),\gamma\rangle, (m,n\in\A^{**})$. Let ${\overline\Phi}^s:(\A^{**})^2\longrightarrow\mathbb C$ be the symmetrized form of $\overline\Phi$ which by \eqref{standard3} is given by  ${\overline\Phi}^s(m,n)=\frac{1}{2}\left(\overline{\Phi}(m,n)+\overline{\Phi}(n,m)\right), (m,n\in\A^{**})$.  Then the $2$-homogeneous polynomial $\overline{P_\gamma}:\mathcal{A}^{**}\rightarrow \mathbb{C}$ induced by ${\overline\Phi}^s$, that is 
     \begin{equation}\label{P1}
      \overline {P_\gamma}(m)={\overline\Phi}^s(m,m)=\overline\Phi(m,m)=\frac{1}{2}\langle m\square m+m\Diamond m,\gamma \rangle,
     \end{equation}
     is orthogonally additive. Indeed, if  $m\square n=0=n\square m$, then by the assumption  $m\Diamond n+n\Diamond m =0$, and so we have 
     \begin{align*}
     \overline{P_\gamma}(m+n)=\overline{\Phi}(m+n,m+n)&=\overline{P_\gamma}(m)+\overline{P_\gamma}(n)+\frac{1}{2}\langle m\Diamond n+n\Diamond m,\gamma\rangle\\
     &=\overline{P_\gamma}(m)+\overline{P_\gamma}(n).
     \end{align*}
    We claim that $\overline{P_\gamma}$ is not in the standard form. Suppose in the contrary that there exists a $\eta\in A^{***}$ such that, 
    \begin{equation}\label{P2}
   \overline{P_\gamma}(m)=\langle\eta, m\square m\rangle,\quad (m\in\A^{**}).
    \end{equation}
     Then by \eqref{standard} its associated  unique symmetric bilinear form must be of the form  $\Psi(m,n)=\frac{1}{2}\langle\eta, m\square n + n\square m\rangle$, for each $m,n\in\A^{**}$. But this can not be happen when $\eta\in \mathcal{A}^{***}\setminus \mathcal{A}^*$; since in this case,  the map $\cdot\rightarrow \Psi(\cdot,x)$ is  not weak$^*$-continuous when $x\in \mathcal{A}$, whereas the map $\cdot\rightarrow {\overline\Phi}^s(\cdot,x)$ is weak$^*$-continuous. Therefore we should assume that $\eta\in A^*.$ In this case, if $\eta=\gamma$, then using  \eqref{P1} and \eqref{P2} for $m=m_0$  we  get $\langle m_0\square m_0, \gamma\rangle=\langle m_0\Diamond m_0,\gamma\rangle,$ which contradicts the choice of $\gamma$. The only possible case is $\eta\neq\gamma.$ Take an element $x\in\A$ with $\langle\eta,x\rangle\neq\langle\gamma,x\rangle$ and let   $E$ be the mixed identity of $\A^{**}$  induced by the bounded approximate identity of $\mathcal{A}$. Then we arrive at  ${\overline\Phi}^s(x,E)=\langle\gamma,x\rangle\neq\langle\eta, x\rangle=\Psi(x,E)$, which means that  ${\overline\Phi}^s$ is not equal $\Psi$. Thus $\overline{P_\gamma}$ is not in the standard form and this completes the proof.
 \end{proof}

\section{Homogeneous polynomials on the bidual of dual Banach algebras}
Here we focus on dual Banach algebras. We recall that  a Banach algebra $\A$  is a dual Banach algebra if $A=F^*$  for some Banach space $F$ and $F$ is a   $\A-$submodule of $\A^*$ (however, we do not use the latter property of $F$ here). Then   the bidual $\mathcal{A}^{**}$ has the presentation  $\mathcal{A}^{**}=\mathcal{A} \oplus F^{\perp}$ where $F^{\perp}=\{m\in\A^{**}; m  \ \hbox{is zero on}\ F\}$ is a closed  ideal of $\A^{**}$. Indeed the bidual $(\A^{**},\Box)$ can be identified with the semidirect product Banach algebra $\mathcal{A}\ltimes F^{\perp}$ equipped with the product $(a,m)(b,n)=(ab, an+mb+m\Box n)$ for every $a,b\in \A, m,n\in F^\perp.$  Furthermore, if $\pi$ denotes the adjoint of the canonical embedding of $F$ to $\A^*$, then a direct verification reveals that the map $\pi:\A^{**}\longrightarrow A$ is a  bounded homomorphism onto $\A.$
\medskip

For a dual Banach algebra $\A$, beside the (Aron--Berner) extension $\overline P$  of a $k-$homogeneous polynomial $P:\A\longrightarrow\mathbb C,$ the map   $P\circ \pi:\A^{**}\longrightarrow\mathbb C$  is also a $k-$homogeneous polynomial extending  $P$. Indeed, if $P$ is induced by the $k-$linear form $\Phi:\A^k\longrightarrow\mathbb C,$ then the $k-$linear form associated with $P\circ\pi$ is $\Phi\circ\pi^k:(\A^{**})^k\longrightarrow\mathbb C,$ where $\pi^k: (\A^{**})^k\longrightarrow\A^k$ is the $k-$fold of  the homomorphism $\pi:\A^{**}\longrightarrow\A$. However, it is known (see  \cite{LM}) that the  extensions $\overline P$ and $P\circ\pi$ are not coincide, in general. To improve the results of \cite{LM}, in the next result we investigate the coincidence of these two extensions for  certain  $k-$homogeneous polynomials.
\begin{proposition}\label{equality}
Let $\mathcal{A}$ be a  non-reflexive dual Banach algebra with the predual $F$ and let $\gamma\in\A^*.$ 
\begin{itemize}
\item[i)] If $\A$ is either Arens regular or commutative,  then  $P_\gamma\circ \pi=\overline{P_\gamma}+ Q,$   where $Q$ is a standard $k-$homogeneous polynomial on $\mathcal{A}^{**}.$
  \item[ii)] If $\gamma\in F$, then $P_\gamma\circ\pi=\overline{P_\gamma}$. The converse is also hold in the case where $\A$ is unital. 
\end{itemize}
\end{proposition}
\begin{proof}
 To prove i), we define $\eta:\A^{**}\longrightarrow\mathbb C$  by $\langle \eta, m\rangle=\langle \gamma, \pi(m)-m\rangle$,\  ($m\in \mathcal{A}^{**}$). By  Proposition \ref{aa}, since for each $m\in\A^{**},$ $\overline{P_\gamma}(m)=\langle m^k,\gamma\rangle,$  we get 
	\[\langle \eta, m^k\rangle=\langle \gamma, \pi(m^k)-m^k\rangle=(P_\gamma\circ\pi)(m)-\overline{P_\gamma}(m).\]
	Therefore  $P\circ\pi(m)=\overline{P}(m)+Q(m)$, where $Q(m)=\langle \eta,m^k\rangle,$ for each $m\in\A^{**}.$\medskip

To prove ii), first suppose that $\gamma\in F,$ since  for each $m,n\in \mathcal{A}^{**}$, $m\Diamond n$ and $m\square n$ are equal on $F$ we have $\overline{P_\gamma}(m)=\langle\gamma, m^k\rangle$, for each $m\in\A^{**}$. The assumption $\gamma\in F$ also implies that $\eta=0$ and   the  latter argument confirms that $P_\gamma\circ\pi=\overline{P_\gamma}.$\smallskip

For the converse suppose that $P_\gamma\circ\pi=\overline{P_\gamma}$ and let $1$ be  the identity of $\A$. Let $\Phi$ be the symmetric $k-$linear form  associated with $P_\gamma$ as given in \eqref{standard}. Then  $\Phi\circ\pi^k$ and ${\overline \Phi}^s$ are the symmetric $k-$linear form associated with $P_\gamma\circ\pi$ and $\overline{P_\gamma},$ respectively, and so $\Phi\circ\pi^k={\overline \Phi}^s$. It follows that $\Phi\circ\pi^k$ is weak$^*-$continuous in its first variable when the other variables  fall in $\mathcal{A}$. On the other hand, by \eqref{standard}, $\Phi\circ\pi^k$ has the expression   
$$(\Phi\circ\pi^k)(m_1,\ldots,m_k)=\frac{1}{k!}\sum_{\sigma\in S_k}\langle\gamma, \pi(m_{\sigma(1)})\ldots \pi(m_{\sigma(k)})\rangle\quad (m_1,\ldots,m_k\in\A^{**}),$$
form which  we get that $(\Phi\circ\pi^k)(m,1,\ldots,1)=\langle\gamma\circ \pi,m\rangle,$ for each $m\in\A^{**}.$ In particular, the functional 
$m\mapsto\langle\gamma\circ \pi,m\rangle:\A^{**}\longrightarrow\mathbb C$  is weak$^*-$continuous.  Now, to prove $\gamma\in F$, it suffices to show that 
$\langle n, \gamma\rangle=0$ for each $n\in F^{\perp}$.  To this end,  let $n\in F^{\perp}$ and  set $m:=1+n\in \mathcal{A}\oplus F^{\perp}=\mathcal{A}^{**}$. Take a net  $(x_{\alpha})$ in $\mathcal{A}$, weak$^*-$converging to $m$. Then by the weak$^*-$continuity of $\gamma\circ P$, we have:
 \begin{align*}
   \langle m,\gamma\rangle=\lim_{\alpha}\langle x_{\alpha},\gamma\rangle=\lim_{\alpha}\langle\pi(x_{\alpha}),\gamma\rangle=
   =\lim_{\alpha}\langle\gamma\circ \pi , x_{\alpha}\rangle=\langle\gamma\circ \pi,m\rangle
   =\langle\gamma,\pi(m)\rangle=\langle\gamma,1\rangle,
 \end{align*}
 which implies that  $\langle n , \gamma\rangle=0,$ as claimed. 
	\end{proof}

The next result confirms that a dual Banach algebra $\A$ inherits $k-$OA property from its bidual $\A^{**}.$
\begin{theorem}\label{inherits}
For a dual Banach algebra  $\mathcal{A},$ if $\A^{**}$ has $k-$OA property then $\A$ has $k-$OA property.
\end{theorem}
\begin{proof}
Let  $P:\A\longrightarrow\mathbb C$ be an orthogonally additive  $k-$homogeneous polynomial induced by the $k-$linear form $\Phi:\A^k\longrightarrow\mathbb C.$ Then, as discussed just before Proposition \ref{equality},  $P\circ \pi:\A^{**}\longrightarrow\mathbb C$, which induced by the $k-$linear form  $\Phi\circ\pi^k:(\A^{**})^k\longrightarrow\mathbb C,$ is  orthogonally additive on $\A^{**}$.  Since $\A^{**}$ enjoys $k-$OA property,  there exists a $\eta\in \mathcal{A}^{***}$ such that  $(P\circ \pi)(m)=\langle\eta, m^{k}\rangle$, for each $m\in\mathcal{A}^{**}$, and this follows that  $P(x)=(P\circ\pi)(x)=\langle\gamma, x^{k}\rangle$, for each  $x\in \mathcal{A}$, where $\gamma=\eta_{|_\A}$. Thus $\A$ has  $k-$OA property, as claimed.
\end{proof}

To achieve a converse to Theorem \ref{inherits}, in  the following result we provide  some conditions under which $\A^{**}$ enjoys $k-$OA property whenever $\A$ has the same property. 
\begin{theorem}\label{aks}
Let $\mathcal{A}$ be a  dual Banach algebra with a predual $F$ such that   $aF^{\perp}=F^{\perp}a=0$ for each $a\in \mathcal{A}$ and $m\Box m=0$ for each $m\in F^\perp$. If  $\mathcal{A}$ has $k-$OA property, then  $\mathcal{A}^{**}$ has $k-$OA property.
\end{theorem}
\begin{proof}
It is known that    $(\mathcal{A}^{**},\square)$ can be identified with the semidirect product Banach $\mathcal{A}\ltimes F^{\perp}$ equipped with the product $(a,m)(b,n)=(ab, an+mb+m\Box n)\ (a,b\in\A, m,n\in F^\perp).$  Let $P$ be  an orthogonal additive $k-$homogeneous polynomial on $\A^{**}=\mathcal{A}\ltimes F^{\perp}$. Since  $(a,0)(0,m)=0=(0,m)(a,0)$, for each $(a,m)\in \mathcal{A}\ltimes F^{\perp}$, we get $P(a,m)=P(a,0)+P(0,m)$. So $P$ can be written as $P(a,m)=P_{1}(a)+P_{2}(m)$ where $P_{1}(a)=P(a,0)$ and $P_{2}(m)=P(0,m)$. Then  $P_{1}$ is an orthogonal additive $k$-homogeneous polynomials on $\mathcal{A}$. So there exists a $\gamma\in \mathcal{A}^{*}$  such that $P_1(a)=\langle\gamma, a^2\rangle$ for each $a\in\A$. On the other hand, since $P$ is an orthogonally additive $k-$homogeneous polynomial and $m\Box m=0$ \ $(m\in F^\perp)$, we get \[2^kP_2(m)=2^kP(0,m)=P(2(0,m))=P((0,m)+(0,m))=P((0,m))+P((0,m))=2P_2(m),\]
for each $m\in F^\perp,$ and this implies that  $P_2(m)=0$ for each $m\in F^\perp$. We thus get that $P$ can be expressed in the standard form, and so  $\mathcal{A}^{**}$ has $k-$OA property, as claimed.
\end{proof}
As an application of Theorem \ref{aks}, we bring the following corollary illustrating that the bidual  $({\ell^1})^{**}$ of the Banach algebra $\ell^1=\ell^1(\mathbb N)$, equipped with the pointwise product,   has $k-$OA property. It is worth to mention that, since $\ell^\infty$ can be identified with  $C(\beta\mathbb N)$, we get $(\ell^1)^{**}\cong M(\beta\mathbb N),$ the measure algebra on the  Stone-\v{C}ech compactification of $\mathbb N,$ (see \cite{D}).
\begin{corollary}\label{l}
  $(\ell^1)^{**}$ has $k-$OA property.
\end{corollary}
  \begin{proof}
  First  we recall that $\ell^1$, with pointwise product,  is a semisimple commutative, Arens regular dual Banach algebra with predual $c_0$. We further recall that $ac_0^\perp=0=c_0^\perp a$ for each $a\in\ell^1,$ and that $m\Box n=0$ for every $m,n\in c_0^\perp,$ (see {\cite[Example 2.6.22(iii)]{D}}). In view of Theorem \ref{aks}, it remains to show that $\ell^1$ has $k-$OA property. This  fact has been proved in \cite{ILL}, however, we  provide a short direct proof here.

Let $P:\ell^1\rightarrow \mathbb{C}$ be an orthogonally additive $k$-homogeneous polynomial and $\Phi:(\ell^1)^k\rightarrow\mathbb{C}$ be its associated  symmetric $k-$linear form.  Orthogonal additivity of $P$ implies that  if  $xy=0=yx$ for some  $x,y\in\ell^1$, then
  $\Phi(x,\underset{i}{\cdots},x,y,\underset{k-i}{\cdots},y)=0$ for each $i=1,\cdots,k-1$; (see for example, {\cite[Proposition 2.2]{PV}}).
   We  define  $\gamma:\ell^1\rightarrow \mathbb{C}$ by
  \[\langle\gamma, x\rangle=\sum_{i=1}^{\infty}\Phi(x,e_i,\cdots,e_i)\quad (x\in\ell^1), \]
   where for each $i\in\mathbb N$, $e_i\in\ell^1$ has $1$ in its $i$th component and $0$ in  the other components.
  Then  $\gamma\in\ell^\infty$. Indeed, for each
  $x=(x_i)\in\ell^1$ and $i\in\mathbb N$, if we set  $\pi_i(x)=x_ie_i$ and $\tau_i(x)=x-\pi_i(x)$, then $\tau_i(x)e_i=0=e_i\tau_i(x)$ and so  
   \begin{align*}
  \sum_{i=1}^{\infty}|\Phi(x,e_i,\cdots,e_i)|=\sum_{i=1}^{\infty}|\Phi(\pi_i(x)+\tau_i(x),e_i,\cdots,e_i)|
  &=\sum_{i=1}^{\infty}|\Phi(\pi_i(x),e_i,\cdots,e_i)|\\
  &=\sum_{i=1}^{\infty}|x_i||\Phi(e_i,e_i,\cdots,e_i)|\leq \|\Phi\|\|x\|.
  \end{align*}
  We now claim that $P$ is standard by showing that $P=P_\gamma.$ To this end, let $x=(x_i)=\sum_{i=1}^{\infty}\pi_i(x)\in\ell^1$, then we have
   \begin{align*}
 P(x)&=\sum_{i=1}^{\infty}\Phi\big(x\cdots,x,\pi_i(x)\big)=\sum_{i=1}^{\infty}\Phi\big(\pi_i(x)+\tau_i(x),\cdots,\pi_i(x)+\tau_i(x),\pi_i(x)\big)\\
  &=\sum_{i=1}^{\infty}\sum_{j=1}^{k-1}\binom{k-1}{j}\Phi\big(\pi_i(x)\underset{j}{\cdots},\pi_i(x),\tau_i(x),\underset{k-1-j}{\cdots},\tau_i(x),\pi_i(x)\big)\\
  &=\sum_{i=1}^{\infty}\Phi\big(\pi_i(x),\cdots,\pi_i(x)\big)\tag {since $\pi_i(x)\tau_i(x)=\tau_i(x)\pi_i(x)=0$}\\
  &=\sum_{i=1}^{\infty}\Phi(x_i^ke_i,\cdots,e_i)
  =\sum_{i=1}^{\infty}\Phi\big(\sum_{j=1}^{\infty}x_j^ke_j,\cdots,e_i\big)=\sum_{i=1}^{\infty}\Phi(x^k,e_i,\cdots,e_i)=\langle\gamma,x^k\rangle.
  \end{align*}
  This completes the proof.
  \end{proof}

  We close  the paper by posing the following  question which, to the best of our knowledge, seems to
be open.

{\bf Question}: Let $\A$ be a Banach algebra with $k-$OA property which is either commutative or Arens regular. Does $\A^{**}$ have $k-$OA property?

\bibliographystyle{amsplain}

\begin{thebibliography}{99}

\bibitem{ABSV}{\sc J. Alaminos, M. Bre\v{s}ar, \v{S}. \v{S}penno and  A.R. Villena,} \textit{Orthogonally additive polynomials and orthosymmetric maps in Banach algebras with properties $\mathbb A$ and $\mathbb B$,}   Proc. Edinburg. Math. Soc. \textbf{59}  (2016), 559-568.


\bibitem{AEGV}{\sc  J. Alaminos, J. Extremera, M.L.C. Godoy and A.R. Villena},
    \textit{Orthogonally additive polynomials on convolution algebras associated with a compact group}, J. Math. Anal. Appl. \textbf{472} (2019), 285–302.

\bibitem{AEV1}{\sc J. Alaminos, J. Extremera and  A.R. Villena} \textit{Orthogonally additive polynomials on Fourier algebras,} J. Math. Anal. Appl. \textbf{422}  (2015), 72-83.


\bibitem{Ar} {\sc R. Arens}, \textit{The adjoint of a bilinear operation}, Proc. Amer. Math. Soc. \textbf{2} (1951), 839--848.



\bibitem{AB} {\sc R.M.  Aron and P. Berner},  \textit{A Hahn-Banach extension theorem for analytic
mappings}, Bull. Soc. Math. France, \textbf{106} (1978), 3--24.



\bibitem{BLL}{\sc Y. Benyamini, S. Lassalle and  J.G. Llavona}, \textit{Homogeneous orthogonally additive polynomials on Banach lattices}, Bull. London Math. Soc. \textbf{38} (2006), 459-469.

\bibitem{D} {\sc H.G. Dales }, \textit{Banach Algebras and Automatic Continuity}, London Math. Soc. Monograph, Vol. 24, Clarendon Press, Axford, 2000.


\bibitem{DG}{\sc A.M. Davie and T.W. Gamelin},  \textit{A theorem on polynomial-star approximation}, Proc. Amer. Math. Soc. \textbf{106} (1989), 351--356.









\bibitem{ILL} {\sc A. Ibort, P. Linares and J.G. Llavona},  \textit{On the representation of orthogonally
additive polynomials in $\ell_p$}, Publ. RIMS, Kyoto Univ.  \textbf{45}  (2009), 519-524.



\bibitem{KEP} {\sc A.A. Khosravi, H.R. Ebrahimi Vishki and A.M. Peralta},  \textit{Aron--Berner extensions of triple maps with application to the bidual of Jordan Banach triple systems}, ArXiv:1901.01822v2.


   \bibitem{LM}{\sc J.G. Llavona and L.A. Moraes}, \textit{The Aron-Berner extension for polynomials
defined in the dual of a Banach Space}, Publ. RIMS, Kyoto Univ., \textbf{40}  (2004), 221--230.




\bibitem{PPV}{\sc C. Palazuelos, A.M. Peralta and I. Villanueva}, \textit{Orthogonally additive polynomials on C$^*$-algebras,}
Quart. J. Math. \textbf{59}, (2008), 353--374.


\bibitem{PV}{\sc D. Pé\' rez-Garcí\' a and I. Villanueva}, \textit{Orthogonally additive polynomials on spaces of continuous functions,}
J. Math. Anal. Appl. \textbf{306}  (2005), 97--105.




\bibitem{V}{\sc   A.R. Villena,} \textit{Orthogonally additive polynomials on Banach function algebras,} J. Math. Anal. Appl. \textbf{448} (2017), 447--472.


\end{thebibliography}

\end{document}